\documentclass[runningheads]{llncs}
\usepackage{graphicx}
\usepackage[linktocpage,pagebackref,colorlinks,linkcolor=blue,anchorcolor=blue,citecolor=blue,urlcolor=blue]{hyperref}
\usepackage{times}
\usepackage[algo2e,linesnumbered,vlined,ruled]{algorithm2e}
\usepackage{graphics,epsf,epstopdf,subfigure}
\usepackage{comment}
\usepackage{amsxtra,amsfonts,amscd,amssymb,amsthm,bm}
\allowdisplaybreaks % automatic page breaks in math environments
\usepackage{supertabular,multirow}
\usepackage{booktabs}
\usepackage[normalem]{ulem}
\usepackage{exscale}

\usepackage{cleveref}
\usepackage{tikz,array}
\usetikzlibrary{shapes.geometric}
\usetikzlibrary{shapes.arrows}
\usetikzlibrary{calc}
\usetikzlibrary{intersections}

\newcommand{\R}{\mathbb{R}}

\newcommand{\Rnnpp}{\mathbb{R}^{2n\times 2p}}

\newcommand{\fs}{^2_{\mathrm{F}}}
\newcommand{\half}{\frac{1}{2}}
\newcommand{\inv}{^{-1}} % inverse
 
\newcommand{\zz}{^{\top}} % matrix transpose

\newcommand{\symplectic}{\mathrm{Sp}(2p,2n)}
\newcommand{\symplecticgroup}{\mathrm{Sp}(2n)}
\newcommand{\skewset}{{\cal S}_{\mathrm{skew}}}
\newcommand{\symset}{{\cal S}_{\mathrm{sym}}}
\newcommand{\TX}{{\mathrm{T}_{X}}\symplectic}
\newcommand{\NX}{\dkh{\TX}^\perp}
\newcommand{\proj}{\mathcal{P}_X}
\newcommand{\projn}{\proj^\perp}
\newcommand{\proje}{\dkh{\proj}_{\mathrm{e}}}
\newcommand{\projen}{\proje^\perp}

\DeclareMathOperator*{\diag}{diag}

\DeclareMathOperator*{\skewsym}{skew}
\DeclareMathOperator*{\spn}{span}
\DeclareMathOperator*{\tr}{tr}

\newcommand{\dkh}[1]{\left(#1\right)}
\newcommand{\fkh}[1]{\left[#1\right]}
\newcommand{\hkh}[1]{\left\{#1\right\}}
\newcommand{\jkh}[1]{\left\langle#1\right\rangle}
\newcommand{\norm}[1]{\left\|#1\right\|}

\newcommand{\rgrade}[1]{\mathrm{grad}_{\mathrm{e}} f(#1)}				
\graphicspath{ {./images/} }
\definecolor{Gray}{rgb}{0.5,0.5,0.5}
\definecolor{myred}{rgb}{0.7,0,0.1}
\definecolor{mygreen}{rgb}{0,0.6,0.2}
\definecolor{myblue}{rgb}{0.5,0,1}

\begin{document}
\title{Geometry of the symplectic Stiefel manifold endowed with the Euclidean metric\thanks{This work was supported by the Fonds de la Recherche Scientifique -- FNRS and the Fonds Wetenschappelijk Onderzoek -- Vlaanderen under EOS Project no. 30468160.}}
%
%\titlerunning{}

\author{Bin Gao\inst{1} \and
	Nguyen Thanh Son\inst{2} \and
	P.-A. Absil\inst{1} \and
	Tatjana Stykel\inst{3}}
\authorrunning{B. Gao, N. T. Son, P.-A. Absil, and T. Stykel}
\institute{ICTEAM Institute, UCLouvain, 1348 Louvain-la-Neuve, Belgium \email{gaobin@lsec.cc.ac.cn} \and
	Thai Nguyen University of Sciences, Thai Nguyen, Vietnam \and
	Institute of Mathematics, University of Augsburg, Augsburg, Germany}

\maketitle              % typeset the header of the contribution

\begin{abstract}
The symplectic Stiefel manifold, denoted by $\symplectic$, is the set of linear symplectic maps between the standard symplectic spaces $\mathbb{R}^{2p}$ and $\mathbb{R}^{2n}$. When $p=n$, it reduces to the well-known set of $2n\times 2n$ symplectic matrices. We study the Riemannian geometry of this manifold viewed as a Riemannian submanifold of the Euclidean space $\Rnnpp$. The corresponding normal space and projections onto the tangent and normal spaces are investigated. Moreover,  we consider optimization problems on the symplectic Stiefel manifold. We obtain the expression of the Riemannian gradient with respect to the Euclidean metric, which then used in optimization algorithms. Numerical experiments on the nearest symplectic matrix problem and the symplectic eigenvalue problem illustrate  the effectiveness of Euclidean-based algorithms.

\end{abstract}

\section{Introduction}\label{sec:intro}
Let $J_{2m}$ denote the nonsingular and skew-symmetric matrix $\fkh{\begin{smallmatrix}
		0&I_m\\ -I_m & 0
		\end{smallmatrix}}$, where $I_m$ is the $m\times m$ identity matrix and $m$ is any positive integer.
The \emph{symplectic Stiefel manifold}, denoted by
$$\symplectic:=\hkh{X\in\Rnnpp: X\zz J_{2n} X = J_{2p}},$$
is a smooth embedded submanifold of the Euclidean space $\Rnnpp$ ($p\le n$)~~\cite[Proposition~3.1]{gao2020riemannian}. We remove the subscript of $J_{2m}$ and $I_m$ for simplicity if there is no confusion. This manifold was studied in~\cite{gao2020riemannian}: it is closed and unbounded; it has dimension $4np-p(2p-1)$; when $p=n$, it reduces to the \emph{symplectic group}, denoted by $\symplecticgroup$. When $X\in\symplectic$, it is termed as a \emph{symplectic} matrix.

Symplectic matrices are employed in many %different
fields. They are indispensable for finding eigenvalues of (skew-)Hamiltonian matrices~\cite{BennF97,BennF98,BennKM05} and for model order reduction of Hamiltonian systems~\cite{PengM16,BuchBH19}. They appear in Williamson's theorem and the formulation of symplectic eigenvalues of symmetric and positive-definite matrices~\cite{williamson1936algebraic,bhatia2015symplectic,jain_mishra_2020,son2021symplectic}. Moreover, symplectic matrices can be found in the study of optical systems~\cite{fiori2016riemannian} and the optimal control of quantum symplectic gates~\cite{wu2008optimal}. Specifically, some applications can be reformulated as optimization problems on the set of symplectic matrices~\cite{PengM16,son2021symplectic}.

%The symplectic matrices appear in many fields. One application can be found in symplectic-based decompositions. Williamson's theorem~\cite{williamson1936algebraic}, SVD-like decomposition~\cite{xu2005numerical}, symplectic Gram--Schmidt-like algorithm~\cite{salam2005theoretical}, and symplectic Householder transformation~\cite{salam2008optimal} were investigated for producing symplectic matrices. Meanwhile, symplectic matrices play an important role in symplectic eigenvalue problems~\cite{bhatia2015symplectic,son2021symplectic}. 
%%hiroshima2006
%Another class of applications is optimization problem over the set of symplectic matrices, i.e., the symplectic Stiefel manifold. Specifically, it can be found in the study of optical systems~\cite{fiori2016riemannian}, 
%%harris2004averageeye
%optimal control of quantum symplectic gates~\cite{wu2008optimal} and symplectic model order reduction problem~\cite{PengM16,BuchBH19}.
%%AfkhH17

In recent decades, most of studies on the symplectic topic focused on the symplectic group ($p=n$) including geodesics of the symplectic group~\cite{fiori2011solving}, optimality conditions for optimization problems on the symplectic group~\cite{machado2002optimization,wu2010critical,birtea2018optimization}, and optimization algorithms on the symplectic group~\cite{fiori2016riemannian,wang2018riemannian}. However, there was less attention to the geometry of the symplectic Stiefel manifold $\symplectic$. More recently, the Riemannian structure of $\symplectic$ was investigated in~\cite{gao2020riemannian} by endowing it with a new class of metrics called \emph{canonical-like}. 
%In view of its geometry, two structure-preserving updates, quasi-geodesics and symplectic Cayley transform, were proposed. 
This canonical-like metric is different from the standard \emph{Euclidean} metric (the Frobenius inner product in the ambient space $\Rnnpp$)
$$\jkh{X,Y}:=\tr(X\zz Y)\quad \mbox{for}~X,Y\in\Rnnpp,$$
where $\tr(\,\cdot\,)$ is the trace operator. 
A priori reasons to investigate the Euclidean metric on $\symplectic$ are that it is arguably the most natural choice, and that there are specific applications with close links to the Euclidean metric, e.g., the projection onto $\symplectic$ with respect to the Frobenius norm (also known as the nearest symplectic matrix problem)
\begin{equation}\label{eq:nearest}
\min\limits_{X\in\symplectic} \norm{X-A}\fs.
\end{equation} 
Note that this problem does not admit a known closed-form solution for general $A\in\Rnnpp$. 
%Therefore, it is of great interest to explore the Riemannian geometry of $\symplectic$ endowed with the Euclidean metric.

In this paper, we consider the symplectic Stiefel manifold $\symplectic$ as a Riemannian submanifold of the Euclidean space $\Rnnpp$. Specifically, the normal space and projections onto the tangent and normal spaces are derived. As an application, we obtain the Riemannian gradient of any function on $\symplectic$ in the sense of the Euclidean metric. Numerical experiments on  the nearest symplectic matrix problem and the symplectic eigenvalue problem are reported. In addition, numerical comparisons with the canonical-like metric are also presented. We observe that the Euclidean-based optimization methods need fewer iterations than the methods with the canonical-like metric on the nearest symplectic problem, and Cayley-based methods perform best among all the choices.

The rest of paper is organized as follows. In section~\ref{sec:Euclidean}, we study the Riemannian geometry of the symplectic Stiefel manifold endowed with the Euclidean metric. This geometry is further applied to optimization problems on the manifold in section~\ref{sec:optimization}. Numerical results are presented in section~\ref{sec:experiments}. 
%In section~\ref{sec:conclusion}, we draw the conclusion.

\section{Geometry of the Riemannian submanifold $\symplectic$}\label{sec:Euclidean}
In this section, we study the Riemannian geometry of $\symplectic$ equipped with the Euclidean metric.

Given $X\in\symplectic$, let $X_\perp\in\mathbb{R}^{2n\times(2n-2p)}$ be a full-rank matrix such that $\spn(X_{\perp})$ is the orthogonal complement of $\spn(X)$. Then the matrix $\left[XJ\,\,\, JX_{\perp}\right]$ is nonsingular, and every matrix $Y\in\Rnnpp$ can be represented as $Y=XJW+JX_\perp K$,
%\begin{align}\label{eq:tangent-decomp}
%Y=XJW+JX_\perp K,
%\end{align}
where $W\in\R^{2p\times 2p}$ and $K\in\R^{(2n-2p)\times 2p}$; see~\cite[Lemma~3.2]{gao2020riemannian}. The tangent space of $\symplectic$ at $X$, denoted by $\TX$, is given by \cite[Proposition~3.3]{gao2020riemannian}
\begin{subequations}
	\begin{align}
	\TX %&=\{Z\in \Rnnpp: Z\zz  J  X + X\zz J  Z=0\} \label{eq:tangent-1}\\
	&= \{XJW+JX_\perp K: W\in\symset(2p), K\in\R^{(2n-2p)\times 2p}\} \label{eq:tangent-2}\\
	&= \{SJX: S\in\symset(2n)\}, \label{eq:tangent-3}
	\end{align}
\end{subequations}
where $\symset(2p)$ denotes the set of all $2p\times 2p$ real symmetric matrices. These two expressions can be regarded as different parameterizations of the tangent space.

Now we consider the Euclidean metric. Given any tangent vectors $Z_i=XJW_i+JX_\perp K_i$ with $W_i\in\symset(2p)$ and $K_i\in\R^{(2n-2p)\times 2p}$ for $i=1,2$, the standard Euclidean metric is defined as
\begin{align*}
	g_{\mathrm{e}}(Z_1,Z_2) &:= \jkh{Z_1,Z_2} = \tr(Z_1\zz Z_2) \\
	 &~= \tr(W_1\zz J\zz X\zz X J W_2)+\tr(K_1\zz X\zz_\perp X_\perp K_2)  \\
	 &\quad+ \tr(W_1\zz J\zz X\zz JX_\perp K_2) + \tr(K_1\zz X\zz_\perp J\zz X J W_2).
\end{align*}
In contrast with the canonical-like metric proposed in~\cite{gao2020riemannian}
\begin{align*}  \label{eq:rmetric}
g_{\rho,X_\perp}(Z_1,Z_2) & :={\frac{1}{\rho}\, \tr}(W_1\zz W_2)+\tr(K_1\zz K_2) \quad \text{with~} \rho>0,
\end{align*}
$g_{\mathrm{e}}$ has cross terms between $W$ and $K$. Note that $g_{\mathrm{e}}$ is also well-defined when it is extended to $\Rnnpp$. Then the normal space of $\symplectic$ with respect to $g_{\mathrm{e}}$ can be defined as
$$\NX_{\mathrm{e}}:=\hkh{N\in\Rnnpp: g_{\mathrm{e}}(N,Z)=0 \mbox{~for all~} Z\in\TX}.$$
We obtain the following expression of the normal space.

\begin{proposition}\label{proposition:normal-Eucliean}
	Given $X\in\symplectic$, we have
	\begin{equation}\label{eq:normal-Eucliean}
	\NX_{\mathrm{e}} = \hkh{JX\varOmega: \varOmega\in\skewset(2p)},
	\end{equation}
	where $\skewset(2p)$ denotes the set of all $2p\times 2p$ real skew-symmetric matrices.
\end{proposition}
\begin{proof}
	Given any $N=JX\varOmega$ with $\varOmega\in\skewset(2p)$, and $Z=XJW+JX_\perp K\in\TX$ with $W\in\symset(2p)$, we have
	$g_{\mathrm{e}}(N,Z)=\tr(N\zz Z)=\tr(\varOmega\zz W)=0$, where the last equality follows from $\varOmega\zz =-\varOmega$ and $W\zz = W$. Therefore, it yields $N\in\NX_{\mathrm{e}}$. Counting dimensions of $\TX$ and {the} subspace $\hkh{JX\varOmega: \varOmega\in\skewset(2p)}$, i.e., $4np-p(2p-1)$ and $p(2p-1)$, respectively, the expression~\eqref{eq:normal-Eucliean} holds. 
\end{proof}

Notice that $\NX_{\mathrm{e}}$ is different from the normal space with respect to the canonical-like metric $g_{\rho,X_\perp}$, denoted by $\NX$, which has the expression $\hkh{XJ\varOmega: \varOmega\in\skewset(2p)}$, obtained in~\cite{gao2020riemannian}.

The following proposition provides explicit expressions for the 
orthogonal projection onto the tangent and normal spaces with respect to the metric $g_{\mathrm{e}}$, denoted by $\proje$ and $\projen$, respectively.

\begin{proposition}\label{proposition:projection-Euclidean}
	Given $X\in\symplectic$ and $Y\in\Rnnpp$, we have
	\begin{align}\label{eq:project-tangent-Euclidean}
	\proje(Y) &=  Y-JX{\varOmega_{X,Y}} ,\\
	\projen(Y) &=  JX{\varOmega_{X,Y}},\label{eq:project-normal-Euclidean}
	\end{align}
	where ${\varOmega_{X,Y}}\in\skewset(2p)$ is the unique solution of the {Lyapunov} equation with unknown~$\varOmega$ %Sylvester 
	\begin{equation}\label{eq:lyapunov}
		X\zz X\varOmega+\varOmega X\zz X =2 \skewsym(X\zz J\zz Y)
	\end{equation}
	and $\skewsym(A):=\half (A-A\zz)$ denotes the skew-symmetric part of $A$.
\end{proposition}
\begin{proof}
	For any $Y\in\Rnnpp$, in view of~\eqref{eq:tangent-2} and~\eqref{eq:normal-Eucliean}, it follows that
	$$
	\proje(Y) =  XJW_Y+JX_\perp K_Y, \qquad
	\projen(Y) =  JX\varOmega,
	$$
	with $W_Y\in\symset(2p)$, $K_Y\in\mathbb{R}^{(2n-2p)\times 2p}$ and $\varOmega\in\skewset(2p)$. Further, $Y$ can be represented as
	$$
	Y = \proje(Y)+\projen(Y) = XJW_Y+JX_\perp K_Y +JX\varOmega.
	$$
	Multiplying this equation from the left with $X\zz J\zz$, it follows that
	\begin{align*}
	X\zz J\zz Y = W_Y+X\zz X\varOmega.
	\end{align*}
	Subtracting from this equation its transpose and taking into account that $W\zz =W$ and $\varOmega\zz=-\varOmega$, we get the Lyapunov equation \eqref{eq:lyapunov} with unknown $\varOmega$. Since $X\zz X$ is symmetric positive definite, all its eigenvalues are positive, and, hence, equation 
	\eqref{eq:lyapunov} has a~unique solution $\varOmega_{X,Y}$; see \cite[Lemma 7.1.5]{golub2013matrix}. Therefore, the relation~\eqref{eq:project-normal-Euclidean} holds. Finally, \eqref{eq:project-tangent-Euclidean} follows from $\proje(Y)=Y-\projen(Y)$.
\end{proof}

\begin{figure}[htbp]
%	\vspace{-6mm}
	\small
	\centering
	\subfigure[Canonical-like metric]
	{ 
		\begin{tikzpicture}[scale=.4]
		\filldraw[color=cyan!10] (1.5,6) -- (-1,4) -- (9,3) -- (11,5.3) -- (1.5,6);
		
		\coordinate [label=left:$X$] (X) at (5,4.8);
		\coordinate [label=left:$\mathcal{M}$] (M) at (4,1);
		\coordinate [label=left:${\mathrm{T}_{X}}\mathcal{M}$] (T) at (0,5);
		\coordinate [label=left:$\dkh{{\mathrm{T}_{X}}\mathcal{M}}^\perp$] (N) at ($(X)+(2,3.5)$);
		\coordinate [label=90:$Y$] (Y) at (8,6.5);
		\coordinate [label=-90:{$\proj(Y)$}] (F) at ($(X)+(1.9,-0.1)$);
		\coordinate [label=left:{$\projn(Y)$}] (F1) at ($(X)+(0.52*2,0.52*3.5)$);
		
		\draw (X) -- (N);
		\draw[dashed] (X) -- ($(X)-(2.5*0.4,2.5*0.7)$);
		\draw ($(X)-(2.5*0.4,2.5*0.7)$) -- ($(X)-(4*0.4,4*0.7)$);
		\draw[-{stealth[red]},red,thick] (X) --  (F);
		\draw[-{stealth[blue]},blue,thick] (X) --  (F1);
		\draw[->] (X) --  (Y);
		\draw[dashed] (Y) -- (F);
		\draw[dashed] (Y) -- (F1);
		
		\node [fill=black,inner sep=.8pt,circle] at (X) {};
		
		%		\draw [help lines] (-5,0) grid (15,8);
		\draw (0,2)	.. controls (0.5,5) and (3,5) .. (6,0.5);
		\draw (0,2)	.. controls (1,2.5) and (2,3) .. (3.92,3);
		\draw (6,0.5)	.. controls (7.5,1.5) and (9.5,2) .. (12,2);
		\draw (0.5,3.4)	.. controls (1.7,5.5) and (11,8) .. (12,2);
		
		\draw [densely dotted] (1.5,6) -- (-1,4) -- (9,3) -- (11,5.3) -- (1.5,6);			
		\end{tikzpicture}
	}
	%	\quad%	\qquad\quad
	\subfigure[Euclidean metric]
	{ 
		\begin{tikzpicture}[scale=.4]
		\filldraw[color=cyan!10] (1.5,6) -- (-1,4) -- (9,3) -- (11,5.3) -- (1.5,6);
		
		\coordinate [label=left:$X$] (X) at (5,4.8);
		\coordinate [label=left:$\mathcal{M}$] (M) at (4,1);
		\coordinate [label=left:${\mathrm{T}_{X}}\mathcal{M}$] (T) at (0,5);
		\coordinate [label=left:$\dkh{{\mathrm{T}_{X}}\mathcal{M}}_{\mathrm{e}}^\perp$] (N) at ($(X)+(0,3.5)$);
		\coordinate [label=90:$Y$] (Y) at (8,6.5);
		\coordinate [label=-90:{$\proje(Y)$}] (F) at ($(Y)+(0,-2)$);
		\coordinate [label=left:{$\projen(Y)$}] (F1) at ($(X)+(0,2)$);
		
		\draw (X) -- (N);
		\draw[dashed] (X) -- ($(X)+(0,-2.8)$);
		\draw ($(X)+(0,-2.95)$) -- ($(X)+(0,-3.5)$);
		\draw[-{stealth[red]},red,thick] (X) --  (F);
		\draw[-{stealth[blue]},blue,thick] (X) --  (F1);
		\draw[->] (X) --  (Y);
		\draw[dashed] (Y) -- (F);
		\draw[dashed] (Y) -- (F1);
		
		\node [fill=black,inner sep=.8pt,circle] at (X) {};
		
		%		\draw [help lines] (-5,0) grid (15,8);
		\draw (0,2)	.. controls (0.5,5) and (3,5) .. (6,0.5);
		\draw (0,2)	.. controls (1,2.5) and (2,3) .. (3.92,3);
		\draw (6,0.5)	.. controls (7.5,1.5) and (9.5,2) .. (12,2);
		\draw (0.5,3.4)	.. controls (1.7,5.5) and (11,8) .. (12,2);
		
		\draw [densely dotted] (1.5,6) -- (-1,4) -- (9,3) -- (11,5.3) -- (1.5,6);			
		\end{tikzpicture}
	}
	\caption{Normal spaces and projections associated with different metrics\label{fig:projection} {on $\mathcal{M}=\symplectic$}}
\end{figure}
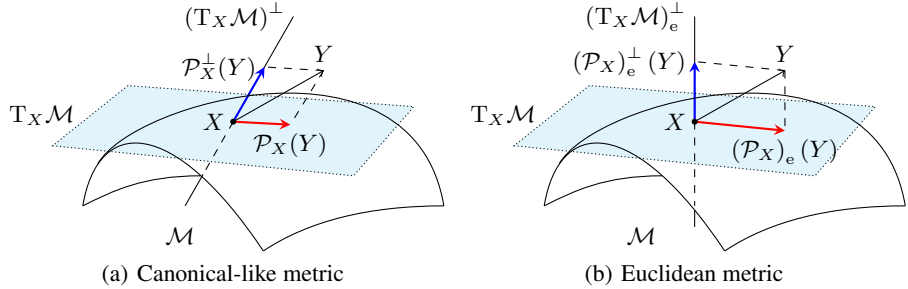

Figure~\ref{fig:projection} illustrates the difference of the normal spaces and projections for the canonical-like metric $g_{\rho,X_\perp}$ and the Euclidean metric $g_{\mathrm{e}}$. Note that projections with respect to the canonical-like metric only require matrix additions and multiplications (see~\cite[Proposition~4.3]{gao2020riemannian}) while one has to solve the Lyapunov equation~\eqref{eq:lyapunov} in the Euclidean case.

The Lyapunov equation \eqref{eq:lyapunov} can be solved using 
%a dense numerical solver such as 
the Bartels--Stewart method~\cite{bartels1972solution}.
%To overcome this difficulty, we suggest to solve this matrix equation as follows. 
Observe that the coefficient matrix $X\zz X$ is symmetric positive definite, and, hence, it has an eigenvalue decomposition $X\zz X = Q\Lambda Q\zz$, where $Q\in\R^{2p\times 2p}$ is orthogonal and $\Lambda = \diag(\lambda_1,\ldots,\lambda_{2p})$ is diagonal with $\lambda_i >0$ for $i=1,\ldots,2p$. Inserting this decomposition into \eqref{eq:lyapunov} and multiplying it from the left and right with $Q\zz$ and $Q$, respectively, we obtain the equation
\begin{align*}\label{eq:Lyapeq2}
\Lambda U +U\Lambda = R
%\Lambda U^T\tilde{\Sigma}U &+ U^T\tilde{\Sigma}U\Lambda\\ \notag
% &= U^T(M^TJ_{2n}\nabla \bar{g}(M) + \nabla \bar{g}(M)^TJ_{2n}M)U.
\end{align*}
with $R %= [g_{ij}]_{i,j=1}^{2p} 
= 2Q\zz\mathrm{skew}(X\zz JY)Q$ and unknown $U%=[u_{ij}]_{i,j=1}^{2p} 
= Q\zz \varOmega Q$. 
The entries of $U$ can then be computed as
$$
u_{ij} = \dfrac{r_{ij}}{\lambda_i + \lambda_j},\quad i,j = 1,\ldots,2p.
$$
Finally, we find $\Omega = QUQ\zz$. The computational cost for matrix-matrix multiplications involved to generate~\eqref{eq:lyapunov} is $O(np^2)$, and $O(p^3)$ for solving this equation.
%Counting the computational cost $O(np^2)$ for matrix-matrix multiplications involved to generate~\eqref{eq:lyapunov} and $O(p^3)$ for solving this equation,  we conclude that computing the projections \eqref{eq:project-tangent-Euclidean} and \eqref{eq:project-normal-Euclidean} has {the computational complexity} of $O(np^2)$ flops.

\section{Application to Optimization}\label{sec:optimization}
In this section, we consider a continuously differentiable real-valued function $f$ on $\symplectic$ and optimization problems on the manifold.
%\begin{equation*}
%\min\limits_{X\in\symplectic} f(X).
%\end{equation*}

The Riemannian gradient of $f$ at $X\in\symplectic$ with respect to the metric $g_{\mathrm{e}}$, denoted by $\rgrade{X}$, is defined as the unique element of $\TX$ that satisfies the condition $g_{\mathrm{e}}\dkh{\rgrade{X},Z} = \mathrm{D}\bar{f}(X)[Z]$ for all $Z\in \TX$,
%$$
%g_{\mathrm{e}}\dkh{\rgrade{X},Z} = \mathrm{D}\bar{f}(X)[Z]\quad
%\text{ for all } Z\in \TX,
%$$
where $\bar{f}$ is a~smooth extension of $f$ around $X$ in $\Rnnpp$, and $\mathrm{D}\bar{f}(X)$ denotes the Fr\'echet derivative of $\bar{f}$ at $X$.  Since $\symplectic$ is endowed with the Euclidean {metric}, the Riemannian gradient can be readily computed by using~\cite[Section~3.6]{absil2009optimization} as follows.

\begin{proposition}\label{proposition:rgrad-Euclidean}
	The Riemannian gradient of {a} function $f: \symplectic\to\R$ with respect to the Euclidean metric $g_{\mathrm{e}}$ has the following form
	\begin{eqnarray}
	\label{eq:rgrad-e}
	\rgrade{X}  =  \proje(\nabla \bar{f}(X)) = \nabla \bar{f}(X)-JX\varOmega_X,
	\end{eqnarray}
	where $\varOmega_X\in\skewset(2p)$ is the unique solution of the {Lyapunov} equation with unknown~$\varOmega$
	$$X\zz X\varOmega+\varOmega X\zz X =2 \skewsym\dkh{X\zz J\zz \nabla \bar{f}(X)},$$
	and $\nabla \bar{f}(X)$ denotes the (Euclidean, i.e., classical) gradient of $\bar{f}$ at $X$.
\end{proposition}

In the case of the symplectic group $\symplecticgroup$, the Riemannian gradient~\eqref{eq:rgrad-e} is equivalent to the formulation in \cite{birtea2018optimization}, where the minimization problem was treated as a~constrained optimization problem in the Euclidean space. We notice that $\varOmega_X$ in~\eqref{eq:rgrad-e} is actually the Lagrangian multiplier of the symplectic constraints; see~\cite{birtea2018optimization}. 

Expression~\eqref{eq:rgrad-e} can be rewritten in the parameterization~\eqref{eq:tangent-2}: it follows from \cite[Lemma~3.2]{gao2020riemannian} that
$$\rgrade{X}=XJW_X+JX_\perp K_X$$ 
with $W_X=X\zz J\zz \rgrade{X}$ and $K_X=\dkh{X\zz_\perp J X_\perp}\inv X\zz_\perp \rgrade{X}$. Moreover, for the purpose of using the Cayley retraction~\cite[Definition~5.2]{gao2020riemannian}, it is essential to rewrite~\eqref{eq:rgrad-e} in the parameterization~\eqref{eq:tangent-3} with $S$ in a factorized form as in~\cite[Proposition~5.4]{gao2020riemannian}. To this end, observe that~\eqref{eq:rgrad-e} is its own tangent projection and use the tangent projection formula of~\cite[Proposition~4.3]{gao2020riemannian} to obtain  $$\rgrade{X}=S_X JX$$ 
with $S_{X} = G_X \rgrade{X} (XJ)\zz + XJ (G_X\rgrade{X})\zz$ and $G_X = I-\frac12 XJX\zz J\zz$.

\section{Numerical Experiments}\label{sec:experiments}
In this section, we adopt the Riemannian gradient~\eqref{eq:rgrad-e} and numerically compare the performance of optimization algorithms  with respect to the Euclidean metric. All experiments are performed on a laptop with 2.7 GHz Dual-Core Intel i5 processor and 8GB of RAM running MATLAB R2016b under macOS 10.15.2. The code that produces the result is available from \href{https://github.com/opt-gaobin/spopt}{https://github.com/opt-gaobin/spopt}. 

\begin{figure}[htpb]
	%This test is generated by the matlab script: "test_EucVScanonical.m" 
	\centering
	\subfigure[F-norm of Riemannian gradient (iteration)]
	{\includegraphics[scale=.325]
		{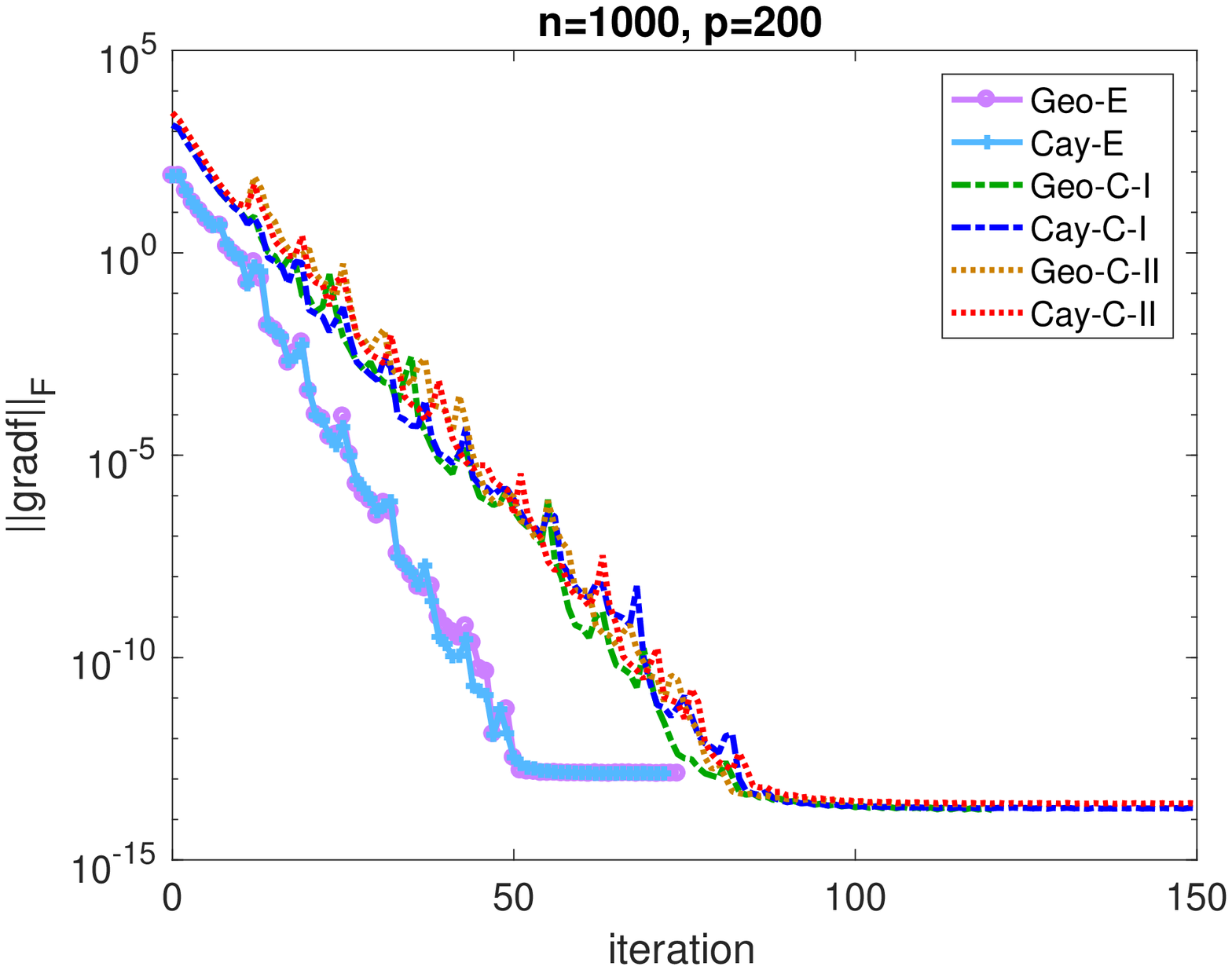}}
	\quad%	\qquad\qquad
	\subfigure[F-norm of Riemannian gradient (time)]
	{\includegraphics[scale=.325]
		{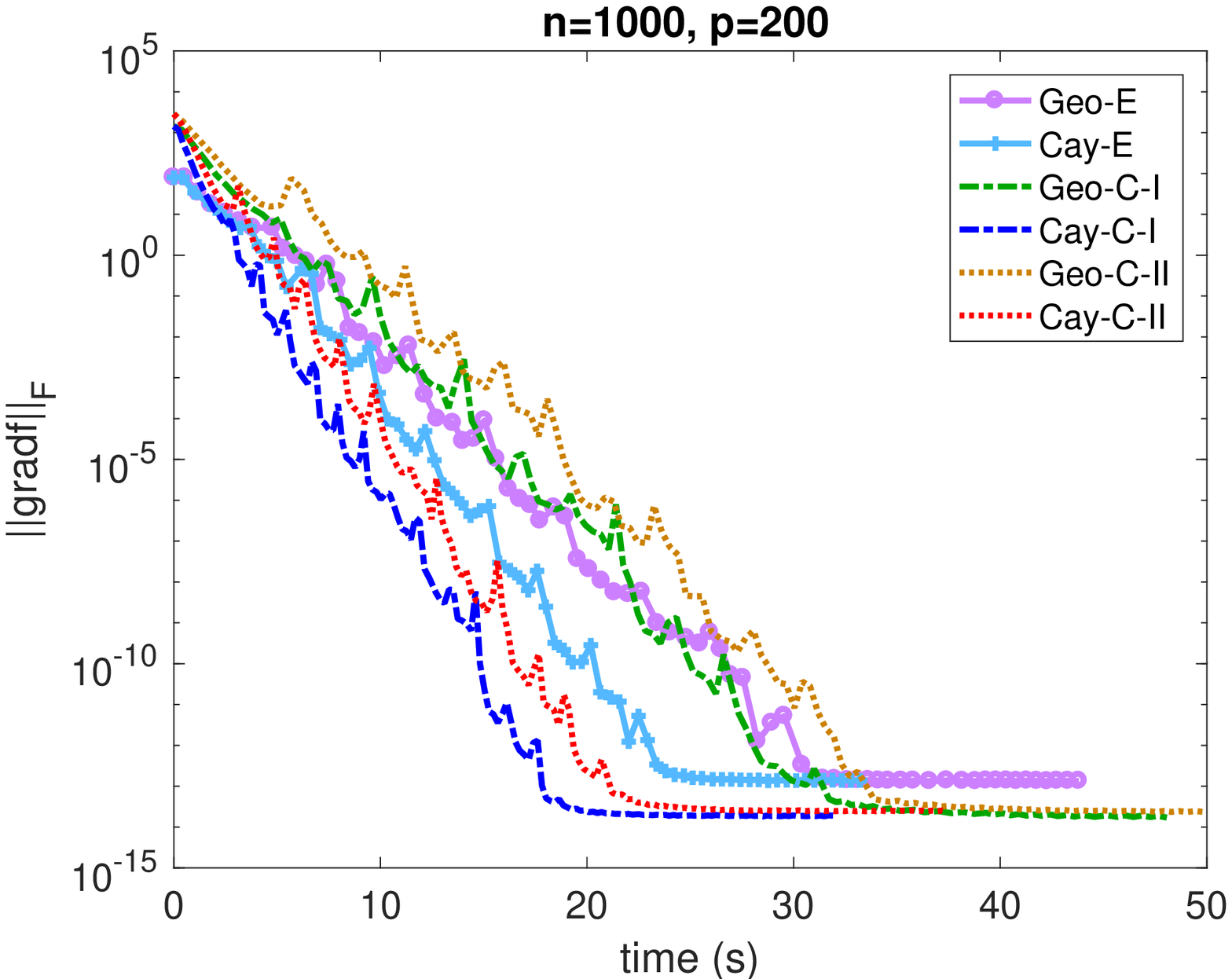}}
	\caption{A comparison of gradient-descent algorithms with different metrics and retractions. Recall that $\mathrm{grad}f$ differs between the ``E'' and ``C'' methods, which explains why they do not have the same initial value.\label{fig:EucVScanonical}}
\end{figure}

First, we consider the optimization problem~\eqref{eq:nearest}. We compare gradient-descent algorithms proposed in~\cite{gao2020riemannian} with different metrics (Euclidean and canonical-like, denoted by ``-E" and ``-C") and retractions (quasi-geodesics and Cayley transform, denoted by ``Geo" and ``Cay"). The canonical-like metric has two formulations, denoted by ``-I" and ``-II", based on different choices of $X_\perp$. Hence, there are six methods involved. The problem generation and parameter settings are in parallel with ones in~\cite{gao2020riemannian}. The numerical results are presented in Figure~\ref{fig:EucVScanonical}. Notice that the algorithms that use the Euclidean metric are considerably superior in the sense of the number of iterations. This can be partly explained by the structure of objective function in~\eqref{eq:nearest}, which is indeed the Euclidean distance. Hence, in this problem the Euclidean metric may be more suitable than other metrics. However, due to their lower computational cost per iteration, algorithms with canonical-like-based Cayley retraction perform best with respect to time among all tested methods, and Cayley-based methods always outperform quasi-geodesics in each setting. 
%In addition, Euclidean-based methods cannot achieve the similar accuracy in contrast with canonical-like settings, while have better performance than type ``-II" setting.

%\begin{figure}[htpb]
%	%This test is generated by the matlab script: "test_speig.m" 
%	\centering
%	\subfigure[errors of symplectic eigenvalues]
%	{\includegraphics[scale=.325]
%		{test_EucVScanonical_eig_vary_n_err}}
%	\quad%	\qquad\qquad
%	\subfigure[normalized residuals]
%	{\includegraphics[scale=.325]
%		{test_EucVScanonical_eig_vary_n_rest}}
%	\subfigure[F-norm of Riemannian gradient (iteration)]
%	{\includegraphics[scale=.325]
%		{test_EucVScanonical_eig_vary_iter}}
%	\quad%	\qquad\qquad
%	\subfigure[F-norm of Riemannian gradient (time)]
%	{\includegraphics[scale=.325]
%		{test_EucVScanonical_eig_vary_timing}}
%	\caption{A comparison of symplectic eigenvalue solvers on varying problem size\label{fig:speig}}
%\end{figure}

The second example is the symplectic eigenvalue problem. We compute the smallest symplectic eigenvalues and eigenvectors of symmetric positive-definite matrices in the sense of Williamson’s theorem; see~\cite{son2021symplectic}. According to the performance in Figure~\ref{fig:EucVScanonical}, we consider ``Cay-E" and ``Cay-C-I" as representative methods. The problem generation and default settings can be found in~\cite{son2021symplectic}. Note that the synthetic data matrix has five smallest symplectic eigenvalues $1,2,3,4,5$. In Table~\ref{tab:speig}, we list the computed symplectic eigenvalues and $1$-norm errors. The results illustrate that our methods are comparable with the structure-preserving eigensolver ``{symplLanczos}" based on a Lanczos procedure~\cite{Amod06}.

\begin{table}[tbhp]
		\caption{Five smallest symplectic eigenvalues of a $1000\times 1000$ matrix computed by different methods}
		\label{tab:speig}
		\begin{center}
			\begin{tabular}{cc c c} \toprule
			 &{symplLanczos} & Cay-E & Cay-C-I\\ \midrule
			&	0.999999999999997 \quad&\quad 1.000000000000000 \quad&\quad 0.999999999999992\\
			&	2.000000000000010  \quad&\quad 2.000000000000010  \quad&\quad 2.000000000000010\\
		     &   3.000000000000014  \quad&\quad 2.999999999999995  \quad&\quad 3.000000000000008\\
			&	4.000000000000004  \quad&\quad 3.999999999999988  \quad&\quad 3.999999999999993\\
			&	5.000000000000016  \quad&\quad 4.999999999999996  \quad&\quad 4.999999999999996 \\ \midrule
			Errors & 4.75e-14 & 3.11e-14 & 3.70e-14 \\ \bottomrule
			\end{tabular}
		\end{center}
\end{table}

%\section{Conclusion}\label{sec:conclusion}
%We have established the geometry of the symplectic Stiefel manifold endowed with the Euclidean metric. Several geometric properties are obtained when the manifold is equipped with the Riemannian structure. In contrast with the canonical-like metric, the Euclidean metric provides a natural way to characterize the geometry, especially in finding the nearest symplectic matrix in the sense of the Frobenius norm distance. These aspects have been verified through numerical examples on different settings. For the future work, the curvature and second-order geometry are of great interest .
% \bibliography{bibfile}

\begin{thebibliography}{10}
	\providecommand{\url}[1]{\texttt{#1}}
	\providecommand{\urlprefix}{URL }
	\providecommand{\doi}[1]{https://doi.org/#1}
	
	\bibitem{absil2009optimization}
	Absil, P.-A., Mahony, R., Sepulchre, R.: Optimization Algorithms on Matrix
	Manifolds. Princeton University Press (2008),
	\url{https://press.princeton.edu/absil}
	
	\bibitem{Amod06}
	Amodio, P.: On the computation of few eigenvalues of positive definite
	{H}amiltonian matrices. Future Generation Computer Systems  \textbf{22}(4),
	403--411 (2006). \doi{10.1016/j.future.2004.11.027}
	
	\bibitem{bartels1972solution}
	Bartels, R.H., Stewart, G.W.: Solution of the matrix equation {$AX+ XB= C$}.
	Commun. ACM  \textbf{15}(9),  820--826 (1972). \doi{10.1145/361573.361582}
	
	\bibitem{BennF97}
	Benner, P., Fassbender, H.: An implicitly restarted symplectic {L}anczos method
	for the {H}amiltonian eigenvalue problem. Linear Algebra Appl.  \textbf{263},
	75--111 (1997). \doi{10.1016/S0024-3795(96)00524-1}
	
	\bibitem{BennF98}
	Benner, P., Fassbender, H.: The symplectic eigenvalue problem, the butterfly
	form, the {SR} algorithm, and the {L}anczos method. Linear Algebra Appl.
	\textbf{275-276},  19--47 (1998). \doi{10.1016/S0024-3795(97)10049-0}
	
	\bibitem{BennKM05}
	Benner, P., Kressner, D., Mehrmann, V.: Skew-{H}amiltonian and {H}amiltonian
	eigenvalue problems: Theory, algorithms and applications. In: Proceedings of
	the Conference on Applied Mathematics and Scientific Computing. pp. 3--39
	(2005). \doi{10.1007/1-4020-3197-1\_1}
	
	\bibitem{bhatia2015symplectic}
	Bhatia, R., Jain, T.: On symplectic eigenvalues of positive definite matrices.
	J. Math. Phys.  \textbf{56}(11),  112201 (2015). \doi{10.1063/1.4935852}
	
	\bibitem{birtea2018optimization}
	Birtea, P., Ca{\c{s}}u, I., Com{\u{a}}nescu, D.: Optimization on the real
	symplectic group. Monatsh. Math.  \textbf{191},  465--485 (2020).
	\doi{10.1007/s00605-020-01369-9}
	
	\bibitem{BuchBH19}
	Buchfink, P., Bhatt, A., Haasdonk, B.: Symplectic model order reduction with
	non-orthonormal bases. Math. Comput. Appl.  \textbf{24}(2) (2019).
	\doi{10.3390/mca24020043}
	
	\bibitem{fiori2011solving}
	Fiori, S.: Solving minimal-distance problems over the manifold of
	real-symplectic matrices. SIAM J. Matrix Anal. Appl.  \textbf{32}(3),
	938--968 (2011). \doi{10.1137/100817115}
	
	\bibitem{fiori2016riemannian}
	Fiori, S.: A {Riemannian} steepest descent approach over the inhomogeneous
	symplectic group: Application to the averaging of linear optical systems.
	Appl. Math. Comput.  \textbf{283},  251--264 (2016).
	\doi{10.1016/j.amc.2016.02.018}
	
	\bibitem{gao2020riemannian}
	Gao, B., Son, N.T., Absil, P.-A., Stykel, T.: {Riemannian} optimization on the
	symplectic {Stiefel} manifold. arXiv preprint arXiv:2006.15226  (2020)
	
	\bibitem{golub2013matrix}
	Golub, G.H., Van~Loan, C.F.: Matrix Computations. Johns Hopkins University
	Press, 4th edn. (2013)
	
	\bibitem{jain_mishra_2020}
	Jain, T., Mishra, H.K.: Derivatives of symplectic eigenvalues and a {L}idskii
	type theorem. Canadian Journal of Mathematics p. 1–29 (2020).
	\doi{10.4153/S0008414X2000084X}
	
	\bibitem{machado2002optimization}
	Machado, L.M., Leite, F.S.: Optimization on quadratic matrix {Lie} groups
	(2002), \url{http://hdl.handle.net/10316/11446}
	
	\bibitem{PengM16}
	Peng, L., Mohseni, K.: Symplectic model reduction of {Hamiltonian} systems.
	SIAM J. Sci. Comput.  \textbf{38}(1),  A1--A27 (2016).
	\doi{10.1137/140978922}
	
	\bibitem{son2021symplectic}
	Son, N.T., Absil, P.-A., Gao, B., Stykel, T.: Symplectic eigenvalue problem via
	trace minimization and {Riemannian} optimization. arXiv preprint
	arXiv:2101.02618  (2021)
	
	\bibitem{wang2018riemannian}
	Wang, J., Sun, H., Fiori, S.: A {Riemannian}-steepest-descent approach for
	optimization on the real symplectic group. Math. Meth. Appl. Sci.
	\textbf{41}(11),  4273--4286 (2018). \doi{10.1002/mma.4890}
	
	\bibitem{williamson1936algebraic}
	Williamson, J.: On the algebraic problem concerning the normal forms of linear
	dynamical systems. Amer. J. Math.  \textbf{58}(1),  141--163 (1936).
	\doi{10.2307/2371062}
	
	\bibitem{wu2010critical}
	Wu, R.B., Chakrabarti, R., Rabitz, H.: Critical landscape topology for
	optimization on the symplectic group. J. Optim. Theory Appl.
	\textbf{145}(2),  387--406 (2010). \doi{10.1007/s10957-009-9641-1}
	
	\bibitem{wu2008optimal}
	Wu, R., Chakrabarti, R., Rabitz, H.: Optimal control theory for
	continuous-variable quantum gates. Phys. Rev. A  \textbf{77}(5),  052303
	(2008). \doi{10.1103/PhysRevA.77.052303}
	
\end{thebibliography}
%

\end{document}